\documentclass[a4paper,10pt,leqno]{amsart}
\usepackage{amsmath,amsfonts,amssymb,amsthm,epsfig,epstopdf,url,array}
\usepackage{amsthm}
\usepackage{mathrsfs}
\usepackage{epsfig}
\usepackage{graphicx}

\newtheorem{thm}{Theorem}[section]
\newtheorem{prop}[thm]{Proposition}
\newtheorem{coro}[thm]{Corollary}
\newtheorem{lem}[thm]{Lemma}
\newtheorem{rmk}[thm]{Remark}

\numberwithin{equation}{section}

\newcommand{\N}{{\mathbb{N}}}
\newcommand{\z}{{\mathbb{Z}}}

\newcommand{\floor}[1]{\left\lfloor #1 \right\rfloor}
\newcommand{\ceil}[1]{\left\lceil #1 \right\rceil}
\newcommand{\eq}{\equiv}

\title[Determining universality of $m$-gonal form]{Determining universality of $m$-gonal form with first five coefficients}

\author{Dayoon Park}
\address{Department of Mathematics, The University of Hong Kong, Hong Kong}
\email{pdy1016@hku.hk}
\thanks{}

\begin{document}
\maketitle 

\begin{abstract}
In this paper, we classify the $(a_1,a_2,a_3,a_4,a_5)$ for which the universality of an $m$-gonal form $F_m(\mathbf x)$ whose first $5$ coefficients are $(a_1,a_2,a_3,a_4,a_5)$ is characterized as the representabilitiy of positive integers up to $m-4$ and see its applications.
\end{abstract}

\maketitle
\section{Introduction}
\vskip 0.3cm
For $m \ge 3$ and $x\ge 1$, the {\it x-th $m$-gonal number}
\begin{equation}\label{m-number}P_m(x)=\frac{m-2}{2}x^2-\frac{m-4}{2}x \end{equation}
which is the number of total dots to constitute a regular $m$-gon with $x$ dots for each side was firstly defined in BC 2-th Century.
Since then, many mathematicians have been interested in representation of positive integers by a sum of $m$-gonal numbers.
As a famous example, Fermat conjectured that every positive integer may be written as a sum of at most $m$ $m$-gonal numbers.
Lagrange and Gauss proved his conjecture for $m=4$ and $m=3$ in 1770 and 1796, respectively.
And finally Cauchy proved his conjecture for all $m \ge 3$ in 1813.

By admitting for variable $x$ in (\ref{m-number}) zero and negative integer too, we can naturally generalize the $m$-gonal number.
Actually, the author and collaborators \cite{KP} showed that for $m \ge 10$, $m-4$ is the optimal number for generalized Fermat's Conjecture, i.e., every positive integer may be written as a sum of $m-4$ generalized $m$-gonal numbers.
In order to more generally consider the representation of positive integers by $m$-gonal numbers,
we call a sum of weighted $m$-gonal numbers 
\begin{equation} \label{m form}
F_m(\mathbf x)=a_1P_m(x_1)+\cdots+a_nP_m(x_n)
\end{equation}
where $a_i \in \N$ an {\it $m$-gonal form}.
In this paper, without loss of generality, we always assume that 
\begin{equation}\label{ass}
a_1 \le \cdots \le a_n
\end{equation}
in (\ref{m form}).
For a positive integer $N \in \N$, if 
$$F_m(\mathbf x)=N$$ has an integer solution $\mathbf x \in \z^n$, then we say that the $m$-gonal form
{\it $F_m(\mathbf x)$ represents $N$} and if $F_m(\mathbf x)$ represents every positive integer, then we call $F_m(\mathbf x)$ is {\it universal}.

The universal forms has been paid attention by many mathematicians.
In the mid $19$-th Century, Liouville classified all universal $3$-gonal forms.
At the beginning of the $20$-th Century, Ramanujan found all universal $4$-gonal forms (i.e., diagonal quadratic forms).
In fact, Ramanujan found $55$ forms but one of them is not actually universal.

In effect, to determine whether a form is universal does not seem to be easy.
On the other hand, in 1993, Conway and Schneeberger announced a stunning result on determining the universality of quadratic form.
Well known as {\it $15$-theorem} \cite{15} the result states that the representability of positive integers up to only $15$ characterize the universality of a quadratic form. 
This simple criterion makes to check very easily the universality of classic quadratic forms.
Kane and Liu \cite{BJ} showed that such a finiteness theorem holds for any $m$-gonal form, i.e., there is (unique and minimal) $\gamma_m$ for which an $m$-gonal form is universal if it represents every positive integer up to $\gamma_m$.
Since the smallest (generalized) $m$-gonal number except $0$ and $1$ is $m-3$, the $\gamma_m$ would be obviously greater than or equal to $m-4$.
In \cite{BJ}, they also firstly questioned about the growth of asymptotically increasing $\gamma_m$ and proved that
$$m-4 \le \gamma_m \ll m^{7+\epsilon}.$$ 
Kim and the author \cite{KP'} optimally improved their result by showing the growth of $\gamma_m$ is exactly linear on $m$, i.e., by proving the following theorem.

\begin{thm} \label{c}
There exists an absolute constant $C>0$ such that $$m-4 \le \gamma_m \le C(m-2)$$
for any $m \ge 3$.
\end{thm}
\begin{proof}
See \cite{KP'}.
\end{proof}
Theorem \ref{c} does not offer the specific $C$.
Even though the proof of Theorem \ref{c} in \cite{KP'} is given by constructing $C$, constructed $C$ would irresponsibly huge, making a big difference with the optimal $C$.
And to find exact $\gamma_m$ for each $m$ would not be easy.
So for the $m$ for which $\gamma_m$ is unknown, to determine the universality of an $m$-gonal form would be still tough topic in general.

Now for
\begin{equation}\label{esc}
\begin{cases}a_1=1 \\ a_{i-1} \le a_i \le a_1+\cdots+a_{i-1}+1 \end{cases}
\end{equation}
and $m-4>a_1+\cdots+a_i$,
let 
\begin{equation}
\gamma_{m,(a_1,\cdots,a_i)}
\end{equation} 
be the minimal integer such that 
an $m$-gonal form whose first $i$ coefficients are $(a_1,\cdots,a_i)$
$$F_m(\mathbf x)=a_1P_m(x_1)+\cdots+a_nP_m(x_n)$$
with $a_1\le \cdots \le a_n$ is universal if and only if it represents every positive integer up to $\gamma_{m,(a_1,\cdots,a_i)}$.
Note that the conditions (\ref{esc}) would be essential for the existence of $\gamma_{m,(a_1,\cdots,a_i)}$ under the assumption (\ref{ass}).
Even though, to determine the universality of arbitrary $m$-gonal form would be difficult, if $\gamma_{m,(a_1,\cdots,a_i)}$ is known, then
one may easily determine the universality of an $m$-gonal form having its first $i$ coefficients as $(a_1,\cdots,a_i)$.
By the definition of $\gamma_{m,(a_1,\cdots,a_i)}$, we immediately have that 
$$m-4 \le \gamma_{m,(a_1,\cdots,a_i)} \le \gamma_m.$$
In effect, $\gamma_m$ is strictly greater than $m-4$.
Although $\gamma_m$ is strictly greater than $m-4$, is there $(a_1,\cdots, a_i)$ for which
$$\gamma_{m,(a_1,\cdots,a_i)}=m-4 ? $$
In this paper, we consider $(a_1,a_2,a_3,a_4,a_5)$ for which $$\gamma_{m,(a_1,a_2,a_3,a_4,a_5)}=m-4$$
for any $m-4>a_1+a_2+a_3+a_4+a_5$ (or at least for almost all $m$). 
The result makes determination of the universality of $m$-gonal form $\sum_{k=1}^na_kP_m(x_k)$ with $\gamma_{m,(a_1,a_2,a_3,a_4,a_5)}=m-4$ easy.
Moreover, one may instantly determine the minimal rank of some specific type of $m$-gonal forms to be universal by applying the results.
Actually, in Section 4, as its applications, we fill out some remaining answers in \cite{KP} and deal with another example.

\vskip 0.5em

In this paper, we use the arithmetic theory of quadratic forms.
A  {\it (quadratic) $\z$-lattice} is a $\z$-module $L$ equipped with a quadratic form $Q$ on $L$ with $Q(L) \subseteq \z$ and its corresponding symmetric bilinear form $B(\mathbf x,\mathbf y)=\frac{Q(\mathbf x+\mathbf y)-Q(\mathbf x)-Q(\mathbf y)}{2}$.
Conventionally, we identify a $\z$-lattice $L=\z\mathbf x_1+\cdots+\z \mathbf x_n$ with its Gram-matrix $\begin{pmatrix}B(\mathbf x_i,\mathbf x_j)\end{pmatrix}_{n \times n}$.
In this paper, for convenience of notation, we will use the notation $[Q(\mathbf x_1),B(\mathbf x_1,\mathbf x_2),Q(\mathbf x_2)]$ instead of $\begin{pmatrix}B(\mathbf x_i,\mathbf x_j)\end{pmatrix}_{2 \times 2}$.
On the other words, $[Q(\mathbf x_1),B(\mathbf x_1,\mathbf x_2),Q(\mathbf x_2)]$ identify with the binary quadratic $\z$-lattice $$Q(\mathbf x_1)X^2+2B(\mathbf x_1,\mathbf x_2)XY+Q(\mathbf x_2)Y^2.$$
In order to avoid confusing of the readers, I would like to note that someone may use the different notation $[Q(\mathbf x_1),2B(\mathbf x_1,\mathbf x_2),Q(\mathbf x_2)]$ for $Q(\mathbf x_1)X^2+2B(\mathbf x_1,\mathbf x_2)XY+Q(\mathbf x_2)Y^2$.
When $\begin{pmatrix}B(\mathbf x_i,\mathbf x_j)\end{pmatrix}_{n \times n}$ is diagonal, i.e., $B(\mathbf x_i,\mathbf x_j)=0$ for any $i \not=j$, we write $\left<Q(\mathbf x_1),\cdots,Q(\mathbf x_n)\right>$ instead of $\begin{pmatrix}B(\mathbf x_i,\mathbf x_j)\end{pmatrix}_{n \times n}$.
For $m$-gonal form $a_1P_m(x_1)+\cdots+a_nP_m(\mathbf x)$, we employ the simple notation $\left<a_1,\cdots,a_n\right>_m$.
Any unexplained notation and terminology can be found in \cite{O1} and \cite{O}.

With the aid of the arithmetic theory of quadratic form, in order to verify the representability of a positive integer by an $m$-gonal form, we examine the representability of a binary quadratic form by a diagonal quadratic form.
Note that
\begin{align*}
A(m-2)+B & =a_1P_m(x_1)+\cdots+a_nP_m(x_n) \\
&=\frac{(m-2)}{2}(\sum_{k=1}^na_kx_k^2-\sum_{k=1}^na_kx_k)+\sum_{k=1}^na_kx_k
\end{align*}
would be solvable if
\begin{equation}\label{system}
\begin{cases}
a_1x_1+\cdots+a_nx_n=B\\
a_1x_1^2+\cdots+a_nx_n^2=2A+B
\end{cases}
\end{equation}
is solvable over $\z$.
If (\ref{system}) is solvable over $\z$, then the binary quadratic $\z$-lattice $$[a_1+\cdots+a_n, B, 2A+B]$$ would be represented by the diagonal quadratic $\z$-lattice $$\left<a_1,\cdots,a_n\right>$$
as follows
$$\begin{pmatrix}a_1+\cdots+a_n & B \\ B & 2A+B\end{pmatrix}=
\begin{pmatrix}1 & 1 & \cdots & 1 \\ x_1 & x_2 & \cdots & x_n\end{pmatrix}
\begin{pmatrix}a_1 & 0 & \cdots & 0 \\ 0 & a_2 & \cdots & 0 \\ \vdots & \vdots & \ddots & \vdots \\ 0 & 0 & \cdots & a_n \end{pmatrix}
\begin{pmatrix}1 & x_1 \\ 1 & x_2 \\ \vdots & \vdots \\1 & x_n \end{pmatrix}.$$
Following the above observation, there would be likely to be a connection between the representability of a positive integer by an $m$-gonal form and the representability of a binary quadratic $\z$-lattice by a diagonal quadratic $\z$-lattice.
In this paper (especially in Section 3), we verify the representability of positive integers by an $m$-gonal form by looking into the association between the representability of a positive integer by an $m$-gonal form and the representability of a binary quadratic $\z$-lattice by a diagonal quadratic $\z$-lattice.
\vskip 1em

\begin{section}*{Acknowledgements}
The author should like to express gratitude to Ben Kane for meticulous reading and many valuable comments which corrected wrong parts and typos and lead a lot of improvement. 
\end{section}

\section{Candidates for $\gamma_{m,(a_1,\cdots,a_5)}=m-4$}

\vskip 1em

In this section, we sort out the candidates $(a_1,a_2,a_3,a_4,a_5)$ for which $$\gamma_{m,(a_1,a_2,a_3,a_4,a_5)}=m-4$$ for any $m-4>a_1+a_2+a_3+a_4+a_5$ (or at least for almost all $m$).
Note that the all $(a_1,a_2,a_3,a_4)$ following (\ref{esc}) are
$$\begin{array} {lllll}
(1,1,1,1),& (1,1,1,2),  \!& (1,1,1,3),  \!&(1,1,1,4),\!&    \\
(1,1,2,2),&  (1,1,2,3),  \!&(1,1,2,4),  \!&(1,1,2,5),\!& \\
(1,1,3,3),&  (1,1,3,4),  \!&(1,1,3,5),  \!&(1,1,3,6),\!& \\
(1,2,2,2),& (1,2,2,3),  \!& (1,2,2,4),  \!& (1,2,2,5),\!& (1,2,2,6),   \\
(1,2,3,3),& (1,2,3,4),  \!& (1,2,3,5),  \!& (1,2,3,6),\!& (1,2,3,7),   \\
(1,2,4,4),& (1,2,4,5),  \!& (1,2,4,6),  \!& (1,2,4,7),\!& (1,2,4,8).   \\
\end{array}$$

\vskip 0.5em
\begin{rmk}
\noindent $(1)$ Every $m$-gonal form $\left<a_1,\cdots,a_n\right>_m$ with $$(a_1,a_2,a_3,a_4)=(1,1,1,3)$$ which represents every positive integer up to $m-4$ with $a_1+\cdots +a_n=m-4$ does not represent $$2m-4.$$
For suppose that $$2m-4=a_1P_m(x_1)+\cdots+a_nP_m(x_n)$$ with $(a_1,a_2,a_3,a_4)=(1,1,1,3)$ and $a_1+\cdots +a_n=m-4$ under (\ref{ass}).
Then without loss of generality, we may have that $x_1 \in \{-1,2\}, \ x_2 \in \{-1,0,1,2\},$ and $x_3, x_4,\cdots,x_n \in \{0,1\}$ since the $m$-gonal numbers $P_m(x)$ are bigger than $m$ except $P_m(-1),$ $P_m(0),$  $P_m(1),$ and $P_m(2)$.
For each pair $(x_1,x_2)$ with $x_1 \in \{-1,2\}$ and $x_2 \in \{-1,0,1,2\},$ one may easily show that there is no $(x_3,\cdots,x_n) \in \{0,1\}^{n-2}$ for
$2m-4=a_1P_m(x_1)+\cdots+a_nP_m(x_n)$ under our assumptions.
For the below cases, one may show those similarly with this argument.

\noindent $(2)$ Every $m$-gonal form $\left<a_1,\cdots,a_n\right>_m$ with $$(a_1,a_2,a_3,a_4)=(1,1,1,4)$$ which represents every positive integer up to $m-4$ with $a_1+\cdots +a_n=m-4$ does not represent $$2m-4.$$

\noindent $(3)$ Every $m$-gonal form $\left<a_1,\cdots,a_n\right>_m$ with $$(a_1,a_2,a_3,a_4)=(1,1,2,5)$$ which represents every positive integer up to $m-4$ with $a_1+\cdots +a_n=m-4$ does not represent $$2m-2.$$

\noindent $(4)$ Every $m$-gonal form $\left<a_1,\cdots,a_n\right>_m$ with $$(a_1,a_2,a_3)=(1,1,3)$$ which represents every positive integer up to $m-4$ with $a_1+\cdots +a_n=m-4$ does not represent $$m-1.$$

\noindent $(5)$ Every $m$-gonal form $\left<a_1,\cdots,a_n\right>_m$ with $$(a_1,a_2)=(1,2)$$ which represents every positive integer up to $m-4$ with $a_1+\cdots +a_n=m-4$ does not represent $$m-2.$$
\end{rmk}

\begin{rmk}
From the above observation, we may see that when $(a_1,a_2,a_3,a_4)=$ $(1,1,1,3),$ $(1,1,1,4),$ $(1,1,2,5),$ $(1,1,3,a_4),$ or $(1,2,a_3,a_4)$,
$$\gamma_{m,(a_1,a_2,a_3,a_4,a_5)}>m-4$$
holds
for any $m-4 \ge a_1+a_2+a_3+a_4+a_5+a_5$.
In other words, in order for $(a_1,a_2,a_3,a_4,a_5)$ to have
$$\gamma_{m,(a_1,a_2,a_3,a_4,a_5)}=m-4$$
for any $m-4>a_1+a_2+a_3+a_4+a_5$ (or at least for almost all $m$), $(a_1,a_2,a_3,a_4)$ should be one of
$$(1,1,1,1), \ (1,1,1,2), \ (1,1,2,2), \ (1,1,2,3), \ (1,1,2,4).$$
So now we get the candidates $(a_1,a_2,a_3,a_4,a_5)$
for $$\gamma_{m,(a_1,a_2,a_3,a_4,a_5)}=m-4$$ 
for any $m-4>a_1+a_2+a_3+a_4+a_5$
as
$$\begin{array} {llllll}
(1,1,1,1,1),& (1,1,1,1,2),  \!& (1,1,1,1,3),  \!&(1,1,1,1,4),\!& (1,1,1,1,5),\!&    \\
(1,1,1,2,2),&  (1,1,1,2,3),  \!&(1,1,1,2,4),  \!&(1,1,1,2,5),\!& (1,1,1,2,6), \!&  \\
(1,1,2,2,2),& (1,1,2,2,3),  \!& (1,1,2,2,4),  \!& (1,1,2,2,5),\!& (1,1,2,2,6),\!& (1,1,2,2,7),   \\
(1,1,2,3,3),&  (1,1,2,3,4),  \!&(1,1,2,3,5),  \!&(1,1,2,3,6),\!& (1,1,2,3,7), \!& (1,1,2,3,8),  \\
(1,1,2,4,4),&  (1,1,2,4,5),  \!&(1,1,2,4,6),  \!&(1,1,2,4,7),\!& (1,1,2,4,8), \!& (1,1,2,4,9).  \\
\end{array}$$

\vskip 1em

\noindent $(1)$ Every $m$-gonal form $\left<a_1,\cdots,a_n\right>_m$ with 
$$(a_1,a_2,a_3,a_4,a_5)=(1,1,1,2,5)$$ 
and
$$a_6=11$$
which represents every positive integer up to $m-4$ and $a_1+\cdots +a_n=m-4$ does not represent 
$$5(m-2)+5.$$
Which implies that $\gamma_{m,(1,1,1,2,5)}>m-4$ for all $m-4 \ge 1+1+1+2+5+11+11.$

\noindent $(2)$ Every $m$-gonal form $\left<a_1,\cdots,a_n\right>_m$ with 
$$(a_1,a_2,a_3,a_4,a_5)=(1,1,1,2,6)$$ 
and
$$a_6=11 \text{ and } a_7=23$$
which represents every positive integer up to $m-4$ and $a_1+\cdots +a_n=m-4$ does not represent $$5(m-2)+11.$$
Which implies that $\gamma_{m,(1,1,1,2,6)}>m-4$ for all $m-4 \ge 1+1+1+2+6+11+23+23.$

\noindent $(3)$ Every $m$-gonal form $\left<a_1,\cdots,a_n\right>_m$ with 
$$(a_1,a_2,a_3,a_4,a_5)=(1,1,2,2,6)$$ 
and
$$a_6=6 \text{ and }  a_7 \ge 12$$
which represents every positive integer up to $m-4$ and $a_1+\cdots +a_n=m-4$ does not represent 
$$5(m-2)+6.$$
Which implies that $\gamma_{m,(1,1,2,2,6)}>m-4$ for all $m-4 \ge 1+1+2+2+6+12+12.$

\noindent $(4)$ Every $m$-gonal form $\left<a_1,\cdots,a_n\right>_m$ with 
$$(a_1,a_2,a_3,a_4,a_5)=(1,1,2,2,7)$$ 
and
$$a_6=13 \text{ and } a_7 \ge19$$
which represents every positive integer up to $m-4$ and $a_1+\cdots +a_n=m-4$ does not represent 
$$5(m-2)+13.$$
Which implies that $\gamma_{m,(1,1,2,2,7)}>m-4$ for all $m-4 \ge 1+1+2+2+7+13+19+19.$

\noindent $(5)$ Every $m$-gonal form $F_m(\mathbf x)=\left<a_1,\cdots,a_n\right>_m$ with 
$$(a_1,a_2,a_3,a_4,a_5)=(1,1,2,3,7)$$ 
and
$$a_6=15 $$
which represents every positive integer up to $m-4$ and $a_1+\cdots +a_n=m-4$ does not represent 
$$7(m-2)+7.$$
Which implies that $\gamma_{m,(1,1,2,3,7)}>m-4$ for all $m-4 \ge 1+1+2+3+7+15+15$.
\noindent $(6)$ Every $m$-gonal form $F_m(\mathbf x)=\left<a_1,\cdots,a_n\right>_m$ with 
$$(a_1,a_2,a_3,a_4,a_5)=(1,1,2,3,8)$$ 
and
$$a_6=15 \text{ and } a_7 \ge 23$$
which represents every positive integer up to $m-4$ and $a_1+\cdots +a_n=m-4$ does not represent 
$$7(m-2)+15.$$
Which implies that $\gamma_{m,(1,1,2,3,8)}>m-4$ for all $m-4 \ge 1+1+2+3+8+15+23+23.$
\end{rmk}

\vskip 1em
From the above observations from (1) to (6), we lastly sort out the candidates $(a_1,a_2,a_3,a_4,a_5)$ for 
$$\gamma_{m,(a_1,a_2,a_3,a_4,a_5)}=m-4$$
for any $m-4>a_1+a_2+a_3+a_4+a_5$ (or at least for almost all $m$) as pairs in Table \ref{t}.

\vskip 0.7em

In next section, we show that for each candidate in Table \ref{t} except $(1,1,2,4,8)$,
$$\gamma_{m,(a_1,a_2,a_3,a_4,a_5)}=m-4$$
holds for all $m-4>a_1+a_2+a_3+a_4+a_5$ and
$$\gamma_{m,(1,1,2,4,8)}=m-4$$
for all $m>2 \sqrt{4+8C}+2$ where $C$ is the constant in Theorem \ref{c}. 

Even though the author failed to find the exact bound for $m$ for $\gamma_{m,(1,1,2,4,8)}=m-4$, one may much reduce the restriction $m$ in order for to have $\gamma_{m,(1,1,2,4,8)}=m-4$.

\begin{center}
\begin{table} 
\caption{Candidates for $\gamma_{m,(a_1,a_2,a_3,a_4,a_5)}=m-4$} 
\label{t}
$\begin{array} {llllll}
(1,1,1,1,1),& (1,1,1,1,2),  \!& (1,1,1,1,3),  \!&(1,1,1,1,4),\!& (1,1,1,1,5),&   \\
(1,1,1,2,2),&  (1,1,1,2,3),  \!&(1,1,1,2,4),  \!&  \!&  & \\
(1,1,2,2,2),& (1,1,2,2,3),  \!& (1,1,2,2,4),  \!& (1,1,2,2,5),\!&  &  \\
(1,1,2,3,3),&  (1,1,2,3,4),  \!&(1,1,2,3,5),  \!&(1,1,2,3,6),\!&    & \\
(1,1,2,4,4),&  (1,1,2,4,5),  \!&(1,1,2,4,6),  \!&(1,1,2,4,7),\!& (1,1,2,4,8),\!& (1,1,2,4,9).   \\
\end{array}$
\end{table}
\end{center}

\section{Main results }

\begin{prop} \label{l1}
Let $F_m(\mathbf x)=a_1P_m(x_1)+\cdots+a_nP_m(x_n)$ under (\ref{ass}) represents every positive integer up to $m-4$.
If $a_1P_m(x_1)+\cdots+a_iP_m(x_i)$ (for some $i \le n$) represents the consecutive $a_1+\cdots+a_i+1$ integers $N,\cdots, N+(a_1+\cdots+a_i)$, then 
$F_m(\mathbf x)$ represents the consecutive $m-3$ integers $N,\cdots, N+(m-4)$.
\end{prop}
\begin{proof}
By assumption, for $0\le r \le m-4$, we may take $\mathbf x(r) \in \{0,1\}^n$
for which $$F_m(\mathbf x(r))=r.$$
And then for $\mathbf y(r)$ with $$a_1P_m(y_1(r))+\cdots +a_iP_m(y_i(r))=N+a_1P_m(x_1(r))+\cdots +a_iP_m(x_i(r)),$$
we have that 
$$a_1P_m(y_1(r))+\cdots +a_iP_m(y_i(r))+a_{i+1}P_m(x_{i+1}(r))+\cdots+aP_n(x_n(r))=N+r.$$
\end{proof}

\begin{coro} \label{rmk}
For an $m$-gonal form $$F_m(\mathbf x)=a_1P_m(x_1)+\cdots+a_nP_m(x_n)$$ under (\ref{ass}) which represents every positive integer up to $m-4$, if 
$$a_1P_m(x_1)+\cdots +a_iP_m(x_i)$$
(for some $i \le n$) represents every positive integer of the form of 
$$A(m-2)+B, \ A(m-2)+(B+1), \ \cdots, \ A(m-2)+(B+a_1+\cdots +a_i+1)$$
for a fixed $B \in \z$, then  
$F_m(\mathbf x)$ is universal.
\end{coro}
\begin{proof}
By using Proposition \ref{l1} with 
$$\text{$N=A(m-2)+B+1$,}$$ 
we obtain that $F_m(\mathbf x)$ represents every positive integer in 
$$[A(m-2)+(B+1),A(m-2)+B+(m-3)]$$ yielding that $F_m(\mathbf x)$ represents every positive integer in $$[A(m-2)+B,A(m-2)+B+(m-3)].$$
Since $\{B, B+1, \cdots, B+(m-3)\}$ forms a complete system of residues modulo $m-2$, it completes the proof.
\end{proof}

\vskip 0.5em

In Lemma \ref{ml}, we look into the association between the representability of a positive integer by an $m$-gonal form and the representability of a binary quadratic $\z$-lattice by a diagonal quadratic $\z$-lattice to figure out some type of positive integers by an $m$-gonal form.

\vskip 0.5em
\begin{lem} \label{ml}
\noindent $(1)$ The $m$-gonal form $$\left<1,1,1,1\right>_m$$ represents every positive integer $A(m-2)+B$ such that 
$$\text{$8A+4B-B^2 \not= 4^l(8k+7)$ and $8A+4B-B^2 \ge 0$}$$
for any $l$ and $k$.

\noindent $(2)$ The $m$-gonal form 
$$\left<1,1,1,2\right>_m$$ 
represents every positive integer $A(m-2)+B$ such that 
$$\text{$(A,B) \not\eq (0,0) \pmod 5$ and $10A+5B-B^2 \ge 0$.}$$

\noindent $(3)$ The $m$-gonal form $$\left<1,1,1,4\right>_m$$ represents every positive integer $A(m-2)+B$ such that 
$$\text{$14A+7B-B^2 \not= 4^l(8k+7)$ and $14A+7B-B^2 \ge 0$}$$
for any $l$ and $k$.

\noindent $(4)$ The $m$-gonal form 
$$\left<1,1,2,2\right>_m$$ 
represents every positive integer $A(m-2)+B$ such that  
$$\text{$12A+6B-B^2 \not= 4^l(8k+7)$ and $12A+6B-B^2 \ge 0$}$$
for any $l$ and $k$.

\noindent $(5)$ The $m$-gonal form 
$$\left<1,1,2,3\right>_m$$
represents every positive integer $A(m-2)+B$ such that  
$$\text{$(A,B) \not\eq (0,0) \pmod 7$ and $14A+7B-B^2 \ge 0$.}$$

\noindent $(6)$ The $m$-gonal form 
$$\left<1,1,2,4\right>_m$$
represents every positive integer $A(m-2)+B$ with $16A+8B-B^2 \ge 0$ otherwise 
$$A\eq 1 \pmod2\text{ and }B \eq 4 \pmod8$$
or
$$A\eq 0 \pmod2\text{ and }B \eq 0 \pmod8.$$

\noindent $(7)$ The $m$-gonal form 
$$\left<1,2,2,2\right>_m$$
represents every positive integer $A(m-2)+B$ such that  
$$\text{$14A+7B-B^2 \not=2\cdot4^l(8k+7)$ and $14A+7B-B^2 \ge 0$}$$
for any $l$ and $k$.

\noindent $(8)$ The $m$-gonal form 
$$\left<1,2,2,3\right>_m$$
represents every positive integer $A(m-2)+B$ such that  
$$\text{$B \not\eq 0 \pmod 4$ and $16A+8B-B^2 \ge 0$.}$$
\end{lem}
\begin{proof}
For the proof of (2), (5), (6), and (8), see Lemma 2.3 in \cite{P'}.

\noindent $(1)$ It would be enough to show that there is an integer solution $(x_1,x_2,x_3,x_4) \in \z^4$ for the diophantine system
\begin{equation}\label{1111}\begin{cases}x_1+x_2+x_3+x_4=B \\ x_1^2+x_2^2+x_3^2+x_4^2=2A+B \end{cases}\end{equation}
where $\text{$8A+4B-B^2 \not= 4^l(8k+7)$ and $8A+4B-B^2 \ge 0$}$.
By easy numeration, we may see that there are exactly two different ways 
$$4=\begin{pmatrix}1 & 1 &1 & 1  \\ \end{pmatrix}
\begin{pmatrix}1 & 0 & 0 & 0 \\ 0 & 1 & 0 & 0 \\ 0 & 0 & 1 & 0 \\ 0 & 0 & 0 & 1 \end{pmatrix}
\begin{pmatrix} 1  \\ 1  \\ 1  \\ 1  \end{pmatrix}$$
and
$$4=\begin{pmatrix}2 & 0 &0 & 0  \\ \end{pmatrix}
\begin{pmatrix}1 & 0 & 0 & 0 \\ 0 & 1 & 0 & 0 \\ 0 & 0 & 1 & 0 \\ 0 & 0 & 0 & 1 \end{pmatrix}
\begin{pmatrix} 2  \\ 0  \\ 0  \\ 0  \end{pmatrix}$$
to represent $4$ by $\left<1,1,1,1\right>$ up to isometry.
So we have that a binary quadratic $\z$-lattice $[4,B,2A+B]$ is represented by $\left<1,1,1,1\right>$ if and only if there is an integer solution $(x_1,x_2,x_3,x_4) \in \z^4$ for either 
$$\begin{pmatrix}4 & B \\ B & 2A+B\end{pmatrix}=
\begin{pmatrix}1 & 1 &1 & 1 \\ x_1 & x_2 & x_3 & x_4\end{pmatrix}
\begin{pmatrix}1 & 0 & 0 & 0 \\ 0 & 1 & 0 & 0 \\ 0 & 0 & 1 & 0 \\ 0 & 0 & 0 & 1 \end{pmatrix}
\begin{pmatrix}1 & x_1 \\ 1 & x_2 \\ 1 & x_3 \\1 & x_4 \end{pmatrix}$$
or
$$\begin{pmatrix}4 & B \\ B & 2A+B\end{pmatrix}=
\begin{pmatrix}2 & 0 &0 & 0 \\ x_1 & x_2 & x_3 & x_4\end{pmatrix}
\begin{pmatrix}1 & 0 & 0 & 0 \\ 0 & 1 & 0 & 0 \\ 0 & 0 & 1 & 0 \\ 0 & 0 & 0 & 1 \end{pmatrix}
\begin{pmatrix}2 & x_1 \\ 0 & x_2 \\ 0 & x_3 \\0 & x_4 \end{pmatrix},$$
i.e., there is an integer solution $(x_1,x_2,x_3,x_4) \in \z^4$ for either (\ref{1111})
or 
\begin{equation} \label{1111'}
\begin{cases}2x_1=B \\ x_1^2+x_2^2+x_3^2+x_4^2=2A+B.
\end{cases}
\end{equation}
Note that for $(x_1,x_2,x_3,x_4) \in \z^4$ satisfying (\ref{1111'}),
$$\left(\frac{x_1+x_2-x_3-x_4}{2}, \frac{x_1-x_2+x_3-x_4}{2},\frac{x_1-x_2-x_3+x_4}{2},\frac{x_1+x_2+x_3+x_4}{2} \right)
$$
would be an integer solution for (\ref{1111}) since $x_1+x_2+x_3+x_4 \eq 0 \pmod 2$.
So we conclude that the diophantine system (\ref{1111}) has an integer solution $(x_1,x_2,x_3,x_4) \in \z^4$ if and only if the binary quadratic $\z$-lattice 
$[4,B,2A+B]$ is represented by the diagonal quaternary quadratic $\z$-lattice $\left<1,1,1,1\right>$.

On the other hand, since the class number of $\left<1,1,1,1\right>$ is one, 
we may see that 
$[4,B,2A+B]$ is represented by $\left<1,1,1,1\right>$ if and only if $8A+4B-B^2 \not= 4^l(8k+7)$ for any 
non-negative $l$ and $k$ and $8A+4B-B^2 \ge 0$
by considering its local representability.
This completes the proof.

(3) Since there is only one representation $$7=1\cdot 1^2+1\cdot 1^2+1\cdot 1^2+4\cdot 1^2$$ of $7$ by $\left<1,1,1,4\right>$ up to isometry, we may conclude that 
\begin{equation}\label{1114}\begin{cases}x_1+x_2+x_3+4x_4=B \\ x_1^2+x_2^2+x_3^2+4x_4^2=2A+B \end{cases}\end{equation}
has an integer solution $(x_1,x_2,x_3,x_4) \in \z^4$ if and only if the binary quadratic $\z$-lattice $[7,B,2A+B]$ is represented by the diagonal quaternary quadratic $\z$-lattice $\left<1,1,1,4\right>$.
By examining the local representability of $[7, B, 2A+B]$ by $\left<1,1,1,4\right>$, one may obtain this since the class number of $\left<1,1,1,4\right>$ is one.

(4) Since there are exactly two different representations $$6=1\cdot 1^2+1\cdot 1^2+2\cdot 1^2+2\cdot 1^2$$ and $$6=1\cdot 2^2+1\cdot 0^2+2\cdot 1^2+2\cdot 0^2$$ of $6$ by $\left<1,1,2,2\right>$ up to isometry,
we have that the binary quadratic $\z$-lattice $[6,B, 2A+B]$ is represented by the quaternary quadratic $\z$-lattice $\left<1,1,2,2\right>$ if and only if there is an integer solution $(x_1,x_2,x_3,x_4) \in \z^4$ for either  
\begin{equation} \label{1122}
\begin{cases}
x_1+x_2+2x_3+2x_4=B \\
x_1^2+x_2^2+2x_3^2+2x_4^2=2A+B,
\end{cases}
\end{equation}
or
\begin{equation} \label{1122'}
\begin{cases}
2x_1+2x_3=B \\
x_1^2+x_2^2+2x_3^2+2x_4^2=2A+B.
\end{cases}
\end{equation}
Note that for $(x_1,x_2,x_3,x_4) \in \z^4$ satisfying (\ref{1122'}),
$$\left(x_3+x_4,x_3-x_4,\frac{x_1+x_2}{2},\frac{x_1-x_2}{2}\right) \in \z^4$$ would be an integer solution for (\ref{1122}).
So we may conclude that the binary quadratic $\z$-lattice $[6,B, 2A+B]$ is represented by the quaternary quadratic $\z$-lattice $\left<1,1,2,2\right>$ if and only if there is an integer solution $(x_1,x_2,x_3,x_4) \in \z^4$ for (\ref{1122}).
By examining the local representability of the binary quadratic form $[6,B, 2A+B]$ by $\left<1,1,2,2\right>$, we may complete the proof since the class number of $\left<1,1,2,2\right>$ is one.

(7) Since there is only one representation $$7=1\cdot 1^2+2\cdot 1^2+2\cdot 1^2+2\cdot 1^2$$ of $7$ by $\left<1,2,2,2\right>$ up to isometry, we may have that
\begin{equation}\label{1222}\begin{cases}x_1+2x_2+2x_3+2x_4=B \\ x_1^2+2x_2^2+2x_3^2+2x_4^2=2A+B \end{cases}\end{equation}
has an integer solution $(x_1,x_2,x_3,x_4) \in \z^4$ if and only if the binary quadratic $\z$-lattice $[7,B,2A+B]$ is represented by the diagonal quaternary quadratic $\z$-lattice $\left<1,2,2,2\right>$.
By examining the local representability of $[7, B, 2A+B]$ by $\left<1,2,2,2\right>$, one may obtain this since the class number of $\left<1,2,2,2\right>$ is one.
\end{proof}

\vskip 0.5em
\begin{rmk}
In Lemma \ref{ml}, we got some types of positive integers which are represented by an $m$-gonal form of rank $4$ by examining the representability of a binary quadratic form by a quaternary diagonal quadratic form.
We took an advatage of some quaternary diagonal quadratic $\z$-lattice having class number one in the proof of Lemma \ref{ml}.
We would have difficulty to use the method in proof of Lemma \ref{ml} when the rank of an $m$-gonal form is greater than $4$ because
a diagonal quadratic form of rank greater than $4$ would be more likely to have class number bigger than one.
And for a diagonal quadratic form of rank smaller than $4$, there would be too much local restriction in representing a binary quadratic form.
So when it comes to examine the positive integers which are represented by an $m$-gonal form by using the method in the proof of Lemma \ref{ml}, it seems that  the most ideal rank for the $m$-gonal form is $4$, even though an $m$-gonal form of rank $5$ would be apt to represent much more integers than an $m$-gonal form of rank $4$.    
With the assistence of the results in Lemma \ref{ml}, in following lemma, we survey the type of positive integers which are represented by an $m$-gonal form of rank $5$.
\end{rmk}

\vskip 0.5em

\begin{lem} \label{lem1}
\noindent $(1)$ For $(a_1,a_2,a_3,a_4,a_5)$ in Table \ref{t} except $(1,1,1,1,4), (1,1,1,1,5),$ $(1,1,2,2,2),$ and $(1,1,2,4,8)$,
$$\left<a_1,a_2,a_3,a_4,a_5\right>_m$$
represents every positive integer of the form of
$$A(m-2)-1, \ A(m-2), \ A(m-2)+1, \cdots, A(m-2)+\left(\sum_{k=1}^5 a_k\right).$$

\noindent $(2)$ When $(a_1,a_2,a_3,a_4,a_5) = (1,1,1,1,4)$ or $(1,1,1,1,5)$, 
$$\left<a_1,a_2,a_3,a_4,a_5\right>_m$$
represents every positive integer of the form of
$$A(m-2)-2, \ A(m-2)-1, \ A(m-2), \cdots, A(m-2)+\left(\sum_{k=1}^5 a_k-1\right)$$
except $(m-2)-2=m-4$.

\noindent $(3)$ When $(a_1,a_2,a_3,a_4,a_5) = (1,1,2,2,2)$, 
$$\left<a_1,a_2,a_3,a_4,a_5\right>_m$$
represents every positive integer of the form of
$$A(m-2)-2, \ A(m-2)-1, \ A(m-2), \cdots, A(m-2)+\left(\sum_{k=1}^5 a_k-1\right)$$
for $A \not\eq 7 \pmod 8$ except $(m-2)-2=m-4$
and 
$$A(m-2)+B$$
for $A \eq 7 \pmod 8$ and $B \in \{-5,-4,\cdots, 2,3\}\cup \{5,6, \cdots, 12, 13\}$.

\end{lem}
\begin{proof}
(1) This proof contains tedious calculations.
So in this paper, we prove this lemma only for $(a_1,a_2,a_3,a_4,a_5)=(1,1,1,1,2)$.
For the other cases, one may show through similar processings with the case $(a_1,a_2,a_3,a_4,a_5)=(1,1,1,1,2)$ by using Lemma \ref{ml}. 

Now let $(a_1,a_2,a_3,a_4,a_5)=(1,1,1,1,2)$.
By Lemma \ref{ml} (2), the $m$-gonal form
$$\left<a_2,a_3,a_4,a_5\right>_m=\left<1,1,1,2\right>_m$$
represents every positive integer of the form of 
$$A(m-2)+B$$
where $B \in \{-1,1,2,3,4\}$.
When $B=0$ or $5$, the positive integers of the form of $$A(m-2)+B$$ may be represented by 
$$P_m(x_1)+P_m(x_2)+P_m(x_3)+P_m(x_4)+2P_m(x_5)$$ with $x_1=1$ because 
$$P_m(x_2)+P_m(x_3)+P_m(x_4)+2P_m(x_5)$$ 
represents every non-negative integer of the form of $A(m-2)-1$ and $A(m-2)+4$ by Lemma \ref{ml} (2).
By using Lemma \ref{ml} (2) again, we may obtain that
$$\left<a_2,a_3,a_4,a_5\right>_m=\left<1,1,1,2\right>_m$$
represents the positive integers
$$A(m-2)+6$$
for $A \ge 1$ and $6$ is immediately represented by $\left<1,1,1,1,2\right>_m$.
Overall, we may conclude that
$$\left<a_1,a_2,a_3,a_4,a_5\right>_m=\left<1,1,1,1,2\right>_m$$
represents every positive integer of the form of 
$$A(m-2)-1, \ A(m-2), \ A(m-2)+1, \cdots, A(m-2)+\left(\sum_{k=1}^5 a_k\right).$$

(2) First consider $(a_1,a_2,a_3,a_4,a_5)=(1,1,1,1,4)$.\\
 By Lemma \ref{ml} (6), since $$\left<1,1,2,4\right>_m$$ represents every positive integer of the form of
$$A(m-2)+B$$
where $B \in \{-2,-1,1,2,3,5,6,7\}$ except $(m-2)-2$,
we may have that
$$\left<1,1,1,1,4\right>_m$$ also represents every positive integer of the form of
$$A(m-2)+B$$
where $B \in \{-2,-1,1,2,3, 5,6,7\}$ except $(m-2)-2$.
By Lemma \ref{ml} (1) and (3), since $\left<1,1,1,1\right>_m$ and $\left<1,1,1,4\right>_m$ represents every positive integer of the form of $A(m-2)$ where  $8A \not= 4^l(8k+7)$ and $14A \not= 4^l(8k+7)$, respectively, we get that 
$$\left<1,1,1,1,4\right>_m$$ 
represents every positive integer of the form of $A(m-2).$
By Lemma \ref{ml} (1) and (6), since  $\left<1,1,1,1\right>_m$ and $\left<1,1,2,4\right>_m$ represent the positive integers $A(m-2)+4$ for $A \eq 1 \pmod2$ and $A\eq 0 \pmod2$, respectively, we obtain that $$\left<1,1,1,1,4\right>_m$$ 
represents every positive integer of the form of $A(m-2)+4.$
Overall, we obtain that 
$$\left<1,1,1,1,4\right>_m$$
represents every positive integer $$A(m-2)+B$$
where $B \in \{-2,-1,0,\cdots,7\}$  except $(m-2)-2$.

For $(a_1,a_2,a_3,a_4,a_5)=(1,1,1,1,5)$, one may show this similarly with the above by using Lemma \ref{ml} (1).

(3)  Through similar processings with the above, one may see that 
$$\left<1,1,2,2,2\right>_m$$
represents every positive integer $$A(m-2)+B$$
for $A \not\eq 7 \pmod{8}$ and $-2 \le B \le 7$ except $(m-2)-2$  by using Lemma \ref{ml} (4), (6), and (7).

For $A \eq 7 \pmod 8$, 
$$\left<1,1,2,2\right>_m$$ represents the positive integers $$\text{$A(m-2)-4$ and $A(m-2)+12$}$$
by Lemma \ref{ml} (4). 
The $m$-gonal form
$$\left<1,1,2,2,2\right>_m$$ would represent the positive integers $$A(m-2)+B$$
when $A\eq 7 \pmod{8}$ and $-5 \le B \le 13$ with $B \not\eq 4 \pmod 8$
since so does $\left<1,1,2,4\right>_m$ by Lemma \ref{ml} (6).
Therefore we obtain that 
$$\left<1,1,2,2,2\right>_m$$ represents the positive integers $$A(m-2)+B$$
where $A\eq 7 \pmod{8}$ and $B \in \{-5,-4,\cdots, 2,3\}\cup \{5,6, \cdots, 12, 13\}$.
\end{proof}

\vskip 0.5em

\begin{thm} \label{main thm}
\noindent $(1)$ For $(a_1,a_2,a_3,a_4,a_5)$ in Table \ref{t} except $(1,1,2,4,8)$,
we have that
$$\begin{cases}
\gamma_{m,(a_1,a_2,a_3,a_4,a_5)}=m-4 & \text{ when } m-4>a_1+\cdots+a_5\\
\left<a_1,a_2,a_3,a_4,a_5\right>_m \text{ is universal } & \text{ when } m-4 \le a_1+\cdots+a_5.
\end{cases}$$

\noindent $(2)$ $\gamma_{m,(1,1,2,4,8)}=m-4$ for $m>2 \sqrt{4+8C}+2$
where $C$ is the constant in Theorem \ref{c}.

\end{thm}
\begin{proof}
(1) By Proposition \ref{l1} and Remark \ref{rmk}, we may obtain the first argument from Lemma \ref{lem1}.

When $(a_1,a_2,a_3,a_4,a_5) \not= (1,1,2,2,2)$, since
$$\{-1,0,\cdots,a_1+\cdots+a_5\} \text{ and } \{-2,-1,\cdots,a_1+\cdots+a_5-1\}$$
contains a complete system of residues modulo $m-2$, respectively, the second argument directly follows from Lemma \ref{lem1}.

Now consider $(a_1,a_2,a_3,a_4,a_5) = (1,1,2,2,2)$.
Note that $$A(m-2)+4=(A-1)(m-2)+(m+2).$$
For each $m-4 \le 8$ except $m=6$, one may easily show that
$$\left<1,1,2,2,2\right>_m$$
represents $(A-1)(m-2)+(m+2)$ for all $A \eq 7 \pmod 8$ by using Lemma \ref{ml} (6).
This implies that $$\left<1,1,2,2,2\right>_m$$
represents all the positive integers of the form of
$$A(m-2)-2, \ A(m-2)-1, \ A(m-2), \cdots, A(m-2)+\left(\sum_{k=1}^5 a_k-1\right)$$
except $(m-2)-2=m-4$ by Lemma \ref{lem1} (3).
And then since $$\{-2,-1,\cdots, \sum _{k=1}^5 a_k -1\}$$ contains a complete system of residues modulo $m-2$ and $\left<1,1,2,2,2\right>_m$ represents every positive integer up to $m-4$, we obtain the claim for $m \not=6$.
When $m=6$, since $\left<1,1,2\right>_6$ is universal, it is clear.

(2) It would be enough to show that 
\begin{equation} \label{2gamma}
\gamma_{m,(1,1,2,4)}=m-4
\end{equation} 
for $m>2 \sqrt{4+8C}+2$.
Throughout this proof, in order to claim (\ref{2gamma}), we show that an $m$-gonal form $$\left<1,1,2,4,a_5,\cdots, a_n\right>_m$$ which represents every positive integer up to $m-4$ represents every positive integer up to $C(m-2)$.
And then Theorem \ref{c} would yield the claim.

For a positive even integer $A$ in $[5,C]$, since $\left<1,1,2\right>_3$ is universal, we may take $x_1,x_2,x_3 \in \z$ for which
\begin{equation}\label{p3}
P_3(x_1-1)+P_3(x_2-1)+2P_3(x_3-1)=A.
\end{equation}
Since $P_3(-x)=P_3(x-1)$, by changing $x_i-1$ to $-x_i$ in (\ref{p3}) if it is necessary, we may assume that 
$$x_1+x_2+2x_3 \eq 0 \pmod 4.$$
And then for $x_4=0$ or $1$, we have that
\begin{equation} \label{o}
x_1+x_2+2x_3+4x_4 \eq 0 \pmod 8.
\end{equation}
So we may get that
\begin{equation} \label{oo}
\begin{cases}
P_m(x_1)+P_m(x_2)+2P_m(x_3)+4P_m(x_4)=A(m-2)+B \\ 
P_m(-x_1+1)+P_m(-x_2+1)+2P_m(-x_3+1)+4P_m(-x_4+1)=A(m-2)-B+8
\end{cases} \end{equation}
where $B=x_1+x_2+2x_3+4x_4 \eq 0 \pmod 8$ for $(x_1,x_2,x_3,x_4)$ satisfying (\ref{p3}) and (\ref{o}) with $x_4 \in \{0,1\}$ from
$$P_m(x)=(m-2)P_3(x-1)+x.$$
Since either $B$ or $-B+8$ is positive in (\ref{oo}), without loss of generality, let $B>0$.
For $(x_1,x_2,x_3)$ satisfying (\ref{p3}), since
$$2-\sqrt{4+8A} \le x_1+x_2+2x_3 \le 2 + \sqrt{4+8A}$$
holds,
we may see that 
$$\begin{cases}
8 \le B  \le 6 + \sqrt{4+8A}\\
2- \sqrt{4+8A} \le -B+8 \le 0.
\end{cases}$$
On the other hands, by Lemma \ref{ml} (6),
$$\left<1,1,2,4\right>_m$$
would represent $$A(m-2)+r$$
for $$r=B-7,B-6,B-5,B-4,B-3,B-2,B-1,B+1(< 4+\sqrt{16+16A})$$
 and 
$$r=(4-\sqrt{16+16A}<)-B+7,-B+9,-B+10,-B+11,-B+12,-B+13,-B+14,-B+15$$
since $A\eq 0 \pmod2$ and $B \eq 0 \pmod8$.
Overall, we may obtain that 
$$\left<1,1,2,4\right>_m$$
represents the consecutive 9 integers in $$[A(m-2)+(B-7),A(m-2)+(B+1)]$$ and the consecutive 9 integers in $$[A(m-2)+(-B+7),A(m-2)+(-B+15)].$$
By Proposition \ref{l1}, which may conclude that 
$\left<1,1,2,4,a_5,\cdots,a_n\right>_m$
represents all the integers in 
$$[A(m-2)-B+7, A(m-2)-B+7+(m-4)] \cup [A(m-2)+B-7,A(m-2)+B-7+(m-4)].$$
Which implies that $\left<1,1,2,4,a_5,\cdots,a_n\right>_m$ represents all the integers in $$[A(m-2)-1 , A(m-2)+(m-3)]$$
because  $B-7 \le-B+7+(m-4) $ from assumption.

For a positive odd integer $A$ in $[5,C]$, since $\left<1,1,2\right>_3$ is universal, we may have $x_1,x_2,x_3 \in \z$ for which
\begin{equation}\label{p3'}
P_3(x_1-1)+P_3(x_2-1)+2P_3(x_3-1)=A.
\end{equation}
Since $P_3(-x)=P_3(x-1)$, by changing $x_i-1$ to $-x_i$ in (\ref{p3'}) if it is necessary, we may assume that 
$$x_1+x_2+2x_3 \eq 0 \pmod 4.$$
And then in this case, we may take $x_4\in \{0 , 1\}$ for which
\begin{equation} \label{o'}
x_1+x_2+2x_3+4x_4 \eq 4 \pmod 8.
\end{equation}
So we may get that
\begin{equation} \label{oo'}
\begin{cases}
P_m(x_1)+P_m(x_2)+2P_m(x_3)+4P_m(x_4)=A(m-2)+B \\ 
P_m(-x_1+1)+P_m(-x_2+1)+2P_m(-x_3+1)+4P_m(-x_4+1)=A(m-2)-B+8
\end{cases} \end{equation}
where $B=x_1+x_2+2x_3+4x_4 \eq 4 \pmod 8$ for $(x_1,x_2,x_3,x_4)$ satisfying (\ref{p3'}) and (\ref{o'}) with $x_4 \in \{0,1\}$.
In this case, without loss of generality, we assume that $B \ge 4$.
When $B > 4$ one may show that $\left<1,1,2,4,a_5,\cdots,a_n\right>_m$ represents all the integers in $$[A(m-2)-1 , A(m-2)+(m-3)]$$ similarly with the above.
When $B=4$, since $\left<1,1,2,4 \right>_m$ represents every positive integer in $[A(m-2)-1, A(m-2)+11]$ by Lemma \ref{ml} (6).
We may conclude that $\left<1,1,2,4,a_5,\cdots,a_n\right>_m$ represents all the integers in $$[A(m-2)-1 , A(m-2)+(m-3)]$$ by Proposition \ref{l1}.

Since for $1 \le A \le 4$,
$$\left<1,1,2,4\right>_m$$
represents all the integers $A(m-2)+B$
where $-1 \le B \le 8$, by Proposition \ref{l1}, 
$\left<1,1,2,4,a_5,\cdots,a_n\right>_m$
would represents the positive integers less than $4(m-2)+(m-3)$, too.
This completes the proof.
\end{proof}

\vskip 1em

\section{some applications }
In this last section, we see some applications for Theorem \ref{main thm}.

In \cite{KP}, the author and collaborators considered the universal $m$-gonal forms of the type  of 
\begin{equation} \label{equ ell}
\left<\underbrace{1,1,\cdots,1,}_{(r-1)\text{-times}}r,r,\cdots,r\right>_m\end{equation}
for $2 \le r <m-4$. 
Especially, for the minimal rank $\ell_m$ for which $\left<1,1,\cdots,1\right>_m$ is universal and the minimal rank $\ell_{m,r,r-1}$ for which (\ref{equ ell}) is universal,
they claimed Theorem \ref{with B thm}.

\begin{thm} \label{with B thm}
\noindent $(1)$ For $m \not\in \{7,9\}$,
$$ \ell_m=
\begin{cases} m-4 & \text{ if } m \ge 10\\
3& \text{ if } m \in \{3,5,6\}\\
4& \text{ if } m \in \{4,8\}. \\
\end{cases}$$

\noindent $(2)$ For $7\le r <m-3$,
$$\ell_{m,r,r-1}=\ceil{\frac{m-3}{r}}+(r-2).$$

\noindent $(3)$

$$\ell_{m,2,1}=\floor{\frac{m}{2}} \text{ for } m \ge 14, $$
$$\ell_{m,3,2}=\begin{cases}m-2 & \text{ for }m\ge10 \text{ with } m \not\eq 2 \pmod 3,\\
\frac{2m-4}{3} & \text{ for }m\ge14 \text{ with } m \eq 2 \pmod 3,\\
\end{cases}$$
$$\ell_{m,4,3}=\ceil{\frac{m-2}{4}}+2 \text{ for } m \ge 62, $$
$$\ell_{m,5,4}=\ceil{\frac{m-3}{5}}+3 \text{ for } m \ge 78, $$
$$\ell_{m,6,5}=\ceil{\frac{m-3}{6}}+4 \text{ for } m \ge 93. $$
\end{thm}
\begin{proof}
See \cite{KP}.
\end{proof}

\vskip 0.5em
In Theorem \ref{up thm}, we fill out some remaining answers in Theorem \ref{with B thm} by applying Theorem \ref{main thm}.
\vskip 0.5em

\begin{thm} \label{up thm}
\noindent $(1)$ $$ \ell_m=
\begin{cases} m-4 & \text{ if } m \ge 9\\
3& \text{ if } m \in \{3,5,6\}\\
4& \text{ if } m \in \{4,7,8\}. \\
\end{cases}$$

\noindent $(2)$ 
For $5 \le r < m-3$, we have
$$\ell_{m,r,r-1}=\ceil{\frac{m-3}{r}}+(r-2).$$

\noindent $(3)$

$$\ell_{m,2,1}=\floor{\frac{m}{2}} \text{ for } m \ge 12, $$
$$\ell_{m,3,2}=\begin{cases}m-2 & \text{ for }m\ge10 \text{ with } m \not\eq 2 \pmod 3,\\
\frac{2m-4}{3} & \text{ for }m\ge14 \text{ with } m \eq 2 \pmod 3,\\
\end{cases}$$
$$\ell_{m,4,3}=\ceil{\frac{m-2}{4}}+2 \text{ for } m \ge 62.$$
\end{thm}
\begin{proof}
(1)
In virtue of Theorem \ref{with B thm}, it would be enough to prove only for $m=7,9$.
First consider $m=7$.
Note that $\left<1,1,1\right>_7$ does not represent $10$.
So we have that $\ell_7\ge 4$.
By Lemma \ref{ml} (1), $\left<1,1,1,1\right>_m$ represents the integers
$$\begin{cases}
A(m-2)-3 & \text{ for } A \ge 3,\\
A(m-2)-1 & \text{ for } A \ge 1,\\
A(m-2)+1 & \text{ for } A \ge 0,\\
A(m-2)+3 & \text{ for } A \ge 0,\\
A(m-2)+5 & \text{ for } A \ge 1.\\
\end{cases}$$
Since $\{-3,-1,1,3,5\}$ form a complete system of residues modulo $m-2=5$ and 
$$ 2(m-2)-3, \ (m-2)-3, \ 0(m-2)+5 $$
are immediately represented by $\left<1,1,1,1\right>_7$ for $m-2=5$, we have that $\ell_7=4$.

For $m \ge 9$, $\ell_m=m-4$ directly follows from Theorem \ref{main thm} (1) for $(a_1,a_2,a_3,a_4,a_5)=(1,1,1,1,1)$.

(2) This directly follows from Theorem \ref{main thm} (1) for $(a_1,a_2,a_3,a_4,a_5)=(1,1,1,1,1)$ and $(1,1,1,1,5)$.

(3) In virtue of the Theorem \ref{with B thm}, it would be enough to show that
$$\ell_{m,2,1}=\floor{\frac{m}{2}} \text{ for } m = 12, 13 $$
One may notice that in order for $\left<1,2,\cdots,2\right>_m$
to represent $m-2$, its rank should be greater than or aqual to $\floor{\frac{m}{2}}$. 
So we have that $\ell_{m,2,1} \ge \floor{\frac{m}{2}}$.
By Lemma \ref{lem1} (1) for $(a_1,a_2,a_3,a_4,a_5)=(1,1,1,1,1)$, we obtain that
$$\left<2,2,2,2,2\right>_m$$
represents every positive integer of the form of 
$$2A(m-2)+2B$$
where $B \in \{-1,0,1,2,3,4,5\}$.
So we may see that
$$P_m(x_1)+2P_m(x_2)+2P_m(x_3)+2P_m(x_4)+2P_m(x_5)+2P_m(x_6)$$
represents every positive integer of the form of
$$2A(m-2)+(2B+1)$$
with $x_1=1$ where $B \in \{-1,0,1,2,3,4,5\}$,
$$(2A+1)(m-2)+(2B+2)$$
with $x_1=2$ where $B \in \{-1,0,1,2,3,4,5\}$ except $m-2$, and
$$(2A+1)(m-2)+(2B-1)$$
with $x_1=-1$ where $B \in \{-1,0,1,2,3,4,5\}$ except $(m-2)-3$.
Oversall, we may see that
$$P_m(x_1)+2P_m(x_2)+2P_m(x_3)+2P_m(x_4)+2P_m(x_5)+2P_m(x_6)$$
represents every positive integer of the form of 
$$A(m-2)+B$$
where $B \in \{-1,0,1,\cdots, 8,9\}$ except $m-2$.
Since $m-2$ is immediately represented by $\left<1,2,2,2,2,2\right>_m$ and $\{-1,0,1,\cdots,9\}$ contains a complete system of residues modulo $m-2$, we may obtain that $P_m(x_1)+2P_m(x_2)+2P_m(x_3)+2P_m(x_4)+2P_m(x_5)+2P_m(x_6)$ is universal.

\end{proof}
 
\begin{thm} \label{fib}
For $m \ge 3$, having its coefficients as Fibonacci sequence $\{F_i\}$
$$\left<1,1,2,3,5,8,\cdots,F_n\right>_m$$
with $n \ge 5$ and $F_1+\cdots+F_n \ge m-4 $ is an universal $m$-gonal form.
\end{thm}
\begin{proof}
By Theorem \ref{main thm} (1)  for $(a_1,a_2,a_3,a_4,a_5)=(1,1,2,3,5)$, it is clear.
\end{proof}

\vskip 1em

\begin{rmk}
Similarly with Thorem \ref{up thm} and Theorem \ref{fib}, there would be a lot of chances to determine the optimal rank of specific type of $m$-gonal forms to be universal by applying Thoerem \ref{main thm}.
\end{rmk}

\end{document}